\documentclass[letterpaper,leqno,11pt,oneside]{amsart}

\usepackage{amsmath,amsthm,amsfonts,amssymb}
\usepackage{enumitem}
\usepackage[usenames,dvipsnames]{color}
\usepackage{mathtools,comment}
\usepackage[extension=pdf]{hyperref}
\usepackage[height=9.6in,width=5.95in]{geometry}
\usepackage{graphics,graphicx}
\usepackage[final,color,notref,notcite]{showkeys}

\usepackage[style=alphabetic,natbib=true,maxnames=99,isbn=false,doi=false,url=false,firstinits=true,hyperref=auto,arxiv=abs]{biblatex}
\DeclareFieldFormat[article,inbook,incollection,inproceedings,patent,thesis,unpublished]{title}{#1\isdot}
\renewbibmacro{in:}{%
  \ifentrytype{article}{}{%
  \printtext{\bibstring{in}\intitlepunct}}}
\defbibheading{apa}[\refname]{\section*{#1}}

\definecolor{darkblue}{rgb}{0.13,0.13,0.39}
\hypersetup{colorlinks=true,urlcolor=darkblue,citecolor=darkblue,linkcolor=darkblue,%
  pdftitle={Continuum statistics of the Airy2 process}}

\mathtoolsset{showonlyrefs,showmanualtags}

\newtheorem{thm}{Theorem}
\newtheorem{lem}{Lemma}[section]
\newtheorem{prop}[lem]{Proposition}

\theoremstyle{definition}
\newtheorem{rem}[lem]{Remark}
\newtheorem*{rem*}{Remark}

\newcounter{assum}

\newcommand{\I}{{\rm i}}
\newcommand{\pp}{\mathbb{P}}

\newcommand{\ee}{\mathbb{E}}
\newcommand{\rr}{\mathbb{R}}

\newcommand{\aip}{\mathcal{A}_2}
\newcommand{\ch}{\mathcal{H}}
\newcommand{\cb}{\mathcal{B}}

\newcommand{\p}{\partial}
\newcommand{\uno}[1]{\mathbf{1}_{#1}}

\newcommand{\vs}{\vspace{6pt}}

\DeclareMathOperator{\Ai}{Ai}
\DeclareMathOperator{\tr}{tr}
\DeclareMathOperator{\re}{Re}
\DeclareMathOperator{\im}{Im}

\numberwithin{equation}{section}

\let\oldmarginpar\marginpar
\renewcommand\marginpar[1]{\-\oldmarginpar[\raggedleft\footnotesize #1]%
{\raggedright{\small\textsf{#1}}}}

\allowdisplaybreaks[0]

\begin{document}

\title{Continuum statistics of the \texorpdfstring{Airy$_2$}{Airy2} process}

\author{Ivan Corwin}
\address[I.~Corwin]{
  Courant Institute\\
   251 Mercer St., Room 604\\
   New York, NY 10012\\
   USA}
 \email{corwin@cims.nyu.edu}
\author{Jeremy Quastel}
\address[J.~Quastel]{
  Department of Mathematics\\
   University of Toronto\\
   40 St. George Street\\
   Toronto, Ontario\\
   Canada M5S 2E4}
 \email{quastel@math.toronto.edu}
\author{Daniel Remenik}
\address[D.~Remenik]{
  Department of Mathematics\\
  University of Toronto\\
  40 St. George Street\\
  Toronto, Ontario\\
  Canada M5S 2E4
  \newline
  \indent\textup{and}\indent
  Departamento de Ingenier\'ia Matem\'atica\\
  Universidad de Chile\\
  Av. Blanco Encala\-da 2120\\
  Santiago\\
  Chile}
  \email{dremenik@math.toronto.edu}

\maketitle

\begin{abstract} We develop an exact determinantal formula for the probability that the Airy$_2$
  process is bounded by a function $g$ on a finite interval. As an application, we provide a direct proof that
  $\sup(\aip(x)-x^2)$ is distributed as a GOE random variable. Both the continuum formula
  and the GOE result have applications in the study of the end point of an unconstrained
  directed polymer in a disordered environment. We explain Johansson's \cite{johansson}
  observation that the GOE result follows from this polymer interpretation and exact
  results within that field. In a companion paper \cite{mqr} these continuum statistics
  are used to compute the distribution of the endpoint of directed polymers.
\end{abstract}

\section{Introduction}

The Airy${}_2$ process $\aip$ was introduced in \citet{prahoferSpohn} in the study of the
scaling limit of a discrete polynuclear growth (PNG) model.  It is expected to govern the
asymptotic spatial fluctuations in a wide variety of random growth models on a one
dimensional substrate with curved initial conditions, and the point-to-point free energies
of directed random polymers in $1+1$ dimensions (the KPZ universality class).  It also
arises as the scaling limit of the top eigenvalue in Dyson's Brownian motion \cite{dyson}
for the Gaussian Unitary Ensemble (GUE) of random matrix theory (see \cite{andGuioZeit} for more
details).

$\aip$ is defined through its finite-dimensional
distributions, which are given by a Fredholm determinant formula: given
$x_0,\dots,x_n\in\mathbb{R}$ and $t_0<\dots<t_n$ in
$\mathbb{R}$,
\begin{equation}\label{eq:detform}
\mathbb{P}\!\left(\aip(t_0)\le x_0,\dots,\aip(t_n)\le x_n\right) =
\det(I-\mathrm{f}^{1/2}K_{\mathrm{ext}}\mathrm{f}^{1/2})_{L^2(\{t_0,\dots,t_n\}\times\mathbb{R})},
\end{equation}
where we have counting measure on $\{t_0,\dots,t_n\}$ and
Lebesgue measure on $\mathbb{R}$,  $\mathrm f$ is defined on
$\{t_0,\dots,t_n\}\times\mathbb{R}$ by
$
\mathrm{f}(t_j,x)=\uno{x\in(x_j,\infty)}$,
and the {\it extended Airy kernel} \cite{prahoferSpohn,FNH,macedo} is
defined by
\[K_\mathrm{ext}(t,\xi;t',\xi')=
\begin{cases}
\int_0^\infty d\lambda\,e^{-\lambda(t-t')}\Ai(\xi+\lambda)\Ai(\xi'+\lambda), &\text{if $t\ge t'$}\\
\int_{-\infty}^0 d\lambda\,e^{-\lambda(t-t')}\Ai(\xi+\lambda)\Ai(\xi'+\lambda),  &\text{if $t<t'$},
\end{cases}\]
where $\Ai(\cdot)$ is the Airy function. In particular, the one point distribution of
$\aip$ is given by the Tracy-Widom largest eigenvalue distribution for GUE.

K. Johansson \cite{johansson} proved the remarkable fact that
\begin{thm}\label{thm:ai}
  For every $m\in\rr$,
  \begin{equation}\pp\!\left(\sup_{t\in\rr}\big(\aip(t)-t^2\big)\leq m\right)=F_\mathrm{GOE}(4^{1/3}m).
  \footnote{The factor $4^{1/3}$ corrects a minor mistake in Johansson's statement.  See
    Section \ref{sec:indirect} for a discussion.}\label{mainthm}\end{equation}
\end{thm}

Here $F_\mathrm{GOE}$ denotes the Tracy-Widom largest eigenvalue distribution for the
Gaussian Orthogonal Ensemble (GOE) \cite{tracyWidom2}. It also arises as the one point
distribution of the Airy${}_1$ process, which governs the asymptotic spatial
fluctuations in one dimensional random growth models with flat initial conditions, and
the point-to-line free energies of directed random polymers in $1+1$ dimensions.

The proof of Theorem \ref{thm:ai} in \cite{johansson} is indirect, using a functional
limit theorem for the convergence of the PNG model to the Airy${}_2$ process, together
with the connection between the PNG process and a certain last passage percolation model
for which \citet{baikRains} had proved the connection with GOE.  In this article we
develop a method to compute continuum probabilities for the Airy${}_2$ process --which is
to say, compute the probability that the sample paths of the Airy${}_2$ process lie below
a given function on any finite interval. This is then used to provide a direct proof of
Theorem \ref{thm:ai} starting only from determinantal formulas.

Theorem \ref{thm:ai} reflects a universal behaviour seen in a large class of one
dimensional systems (the KPZ universality class starting with flat initial conditions) and
therefore has attracted quite a bit of interest at the physical level.  Much of the recent
work is on finite systems of $N$ nonintersecting random walks, the so-called vicious
walkers \cite{fisher}.  \cite{feierl2,RS1,RS2} obtain various expressions for the
maximum and position of the maximum at the finite $N$ level.  \cite{forrester} uses
non-rigorous methods from gauge theory to obtain the GOE distribution in the large $N$
limit, and furthermore connect the problem to Yang-Mills theory.

Our computation of continuum probabilities starts with the following
(earlier) variant of \eqref{eq:detform} due to \citet{prahoferSpohn},
\begin{multline}
  \label{eq:airyfd}
    \pp\!\left(\aip(t_0)\leq x_0,\dotsc,\aip(t_n)\leq x_n\right)\\
  =\det\!\left(I-K_{\Ai}+\bar P_{x_0}e^{(t_0-t_1)H}\bar P_{x_1}e^{(t_1-t_2)H}\dotsm
   \bar P_{x_n}e^{(t_n-t_0)H}K_{\Ai}\right),
\end{multline}
where $K_{\Ai}$ is the \emph{Airy kernel}
\[K_{\Ai}(x,y)=\int_{-\infty}^0 d\lambda\Ai(x-\lambda)\Ai(y-\lambda),\] $H$ is the
\emph{Airy Hamiltonian} $H=-\p_x^2+x$ and $\bar P_a$ denotes the projection onto the
interval $(-\infty,a]$. Here, and in everything that follows, the determinant means the
Fredholm determinant in the Hilbert space $L^2(\rr)$. The equivalence of
\eqref{eq:detform} and \eqref{eq:airyfd} was derived formally in \cite{prahoferSpohn} and
\cite{prolhacSpohn}. In fact there are some subtleties, because, for example, it is not apriori obvious
that for $s,t>0$, $e^{-sH}$ can be applied to the image of $\bar P_ae^{-tH}$. See
\cite{quastelRemAiry1} for a discussion of the technical details. 

\begin{rem}\label{airyrem}
  The shifted Airy functions are the generalized eigenfunctions of the Airy Hamiltonian,
  as $H\!\hspace{0.05em}\Ai(x-\lambda)=\lambda\!\hspace{0.1em}\Ai(x-\lambda)$. The Airy
  kernel $K_{\Ai}$ is the projection of $H$ onto its negative generalized eigenspace. This
  is seen by observing that if we define the operator $A$ to be the \emph{Airy transform},
  $Af(x):=\int_{-\infty}^\infty dz\Ai(x-z)f(z)$, then $K_{\Ai}=A\bar P_0A^*$.
\end{rem}

Fix $\ell<r$. Given $g\in H^1([\ell,r])$ (i.e. both $g$ and its derivative are in
$L^2([\ell,r])$), define an operator $\Theta^g_{[\ell,r]}$ acting on $L^2(\rr)$ as follows:
$\Theta^g_{[\ell,r]}f(\cdot)=u(r,\cdot)$, where $u(r,\cdot)$ is the solution at time $r$ of the
boundary value problem
 \begin{equation}
\begin{aligned}
  \p_tu+Hu&=0\quad\text{for }x<g(t), \,\,t\in (\ell,r)\\
  u(\ell,x)&=f(x)\uno{x<g(\ell)}\\
  u(t,x)&=0\quad\text{for }x\ge g(t).
\end{aligned}\label{eq:bdval}
\end{equation}
The fact that this problem makes sense for $g\in H^1([\ell,r])$ is easy to prove and can be
seen from the proof of Proposition \ref{prop:theta} below. By taking a fine mesh in $t$
we obtain a continuum version of (\ref{eq:airyfd}):

\begin{thm}\label{thm:aiL}
  \begin{equation}\pp\!\left(\aip(t)\leq g(t)\text{ for
      }t\in[\ell,r]\right)=\det\!\left(I-K_{\Ai}+\Theta^g_{[\ell,r]}  e^{(r-\ell)H}K_{\Ai}\right).\label{eq:basic}
  \end{equation}
 \end{thm}

An expression in terms of determinants of solution operators of boundary value problems
 may not seem very practical.  But in fact one can give an explicit expression for the
 kernel of the operator $\Theta^g_{[\ell,r]}$ in terms of Brownian motion. Let $b(s)$ denote a
 Brownian motion with diffusion coefficient $2$.  By the Feynman-Kac formula,
\[u(r,x)=\ee_{b(\ell)=x}\!\left(f(b(r))e^{-\int_{\ell}^rb(s)ds}\uno{b(s)\leq g(s)\text{ on
    }[\ell,r]}\right).\] The linear potential is removed by a parabolic shift,
\begin{align}
  \Theta^g_{[\ell,r]}f(x)
  &=\ee_{b(\ell)=x}\!\left(f(b(r))e^{-\int_{\ell}^rb(s)ds}\uno{b(s)\leq
      g(s)\text{ on }[\ell,r]}\right)\\
  &=\ee_{b(\ell)=x}\!\left(f(b(r))e^{\ell b(\ell)-rb(r)+(r^3-\ell^3)/3+\int_{\ell}^rsdb(s)-\int_{\ell}^rs^2ds}\uno{b(s)\leq
      g(s)\text{ on }[\ell,r]}\right)\\
  &=\ee_{b(\ell)=x-\ell^2}\!\left(f(b(r)+r^2)e^{\ell(b(\ell)+\ell^2)-r(b(r)+r^2)+(r^3-\ell^3)/3}\uno{b(s)+s^2\leq
      g(s)\text{ on }[\ell,r]}\right),
\end{align}
where in the second equality we used integration by parts and added and subtracted
$(r^3-\ell^3)/3$ and in the third one we used the Cameron-Martin-Girsanov formula.  This
gives

\begin{thm}\label{thm:thetaLgen} Let $\Theta^g_{[\ell,r]}(x,y)$ denote the integral kernel of
  $\Theta^g_{[\ell,r]}$. Then
\begin{multline}\label{eq:ThetaL}
\Theta^g_{[\ell,r]}(x,y)=e^{\ell x-ry+(r^3-\ell^3)/3}\frac{e^{-(x-y)^2/4(r-\ell)}}{\sqrt{4\pi(r-\ell)}}\\
\cdot\pp_{\hat b(\ell)=x-\ell^2,\hat b(r)=y-r^2}\!\left(\hat b(s)\leq g(s)-s^2\text{ on }[\ell,r]\right),
\end{multline}
where the probability is computed with respect to a Brownian bridge $\hat b(s)$ from $x-\ell^2$ at time $\ell$ to
$y-r^2$ at time $r$ and with diffusion coefficient $2$.
\end{thm}
\noeqref{eq:thetaL}

This gives a formula which can be used in applications.  The obvious one is the case
$g(t)= t^2+m$, in which the probability can easily be computed by the reflection principle
(method of images). A second one is the computation of the joint distribution of the max
and argmax of the Airy$_2$ process minus a parabola, which appears in a companion paper
\cite{mqr}. The simple result in the case $g(t)= t^2+m$, setting $-\ell=r=L$, is that
\begin{equation}
\Theta_L:=\Theta^{g(t)=t^2+m}_{[-L,L]}
=\bar P_{m+L^2}e^{-2LH}\bar P_{m+L^2}-\bar P_{m+L^2}R_L\bar P_{m+L^2},\label{eq:thetaL}
\end{equation}
where $R_L$ is the reflection term
\begin{equation}\label{eq:RL}
  R_L(x,y)=\frac{1}{\sqrt{8\pi L}}e^{-(x+y-2m-2L^2)^2/8L-(x+y)L+2L^3/3}.
\end{equation}
The first term in $\Theta_L$ has been reexpressed in terms of the Airy Hamiltonian by
reversing the use of the Cameron-Martin-Girsanov and Feynman-Kac formulas.\noeqref{eq:RL}

To obtain the $L\to \infty$ asymptotics, decompose $\Theta_L$ so as to expose the two limiting terms, as well as a remainder term $\Omega_L$:
\begin{equation}
\Theta_L=e^{-2LH}-R_L+\Omega_L,\label{eq:OmegaL}
\end{equation}
where $\Omega_L=\big(R_L-\bar P_{m+L^2}R_L\bar P_{m+L^2}\big)
-\big(e^{-2LH}-\bar P_{m+L^2}e^{-2LH}\bar P_{m+L^2}\big)$.  In Section \ref{sec:extras2} we will show
that
\begin{lem}\label{eq:omega}
As $L$ goes to infinity,
\begin{equation}
  \widetilde\Omega_L := e^{LH}K_{\Ai}\Omega_Le^{LH}K_{\Ai}\rightarrow 0
\end{equation} in trace norm.
\end{lem}

Referring to \eqref{eq:basic},  we have by the cyclic property of determinants and the
identity $e^{2LH}K_{\Ai}=(e^{LH}K_{\Ai})^2$ that
\begin{equation}\label{eq:basiccyclic}
\pp\!\left(\aip(t)\leq g(t)\text{ for  }t\in[-L,L]\right)=\det\!\left(I-K_{\Ai}+e^{LH}K_{\Ai}
  \Theta_L e^{LH}K_{\Ai}\right).
\end{equation}
Since $e^{LH}K_{\Ai}e^{-2LH}e^{LH}K_{\Ai}=K_{\Ai}$ and due to Lemma \ref{eq:omega}, one
sees that the key point is the limiting behaviour in $L$ of
$e^{LH}K_{\Ai}R_Le^{LH}K_{\Ai}$.  Remarkably, it does not depend on $L$ and gives the
kernel of $F_{\rm GOE}$, thus providing a proof of Theorem
\ref{thm:ai}.\noeqref{eq:basiccyclic}\noeqref{eq:OmegaL}

\begin{prop}\label{thm:goe}
For all $L>0$, 
\begin{equation}
 e^{LH}K_{\Ai}R_Le^{LH}K_{\Ai}=A\bar P_0\hat R\bar P_0A^*,\label{eq:tildeRL}
\end{equation}
where the $A$ is the \emph{Airy transform} (see Remark \ref{airyrem}), and 
\begin{equation}
\hat R(\lambda,\tilde\lambda) :
=2^{-1/3}\Ai(2^{-1/3}(2m-\lambda-\tilde\lambda)).\label{eq:hatRL1}
\end{equation}
Furthermore,
\begin{equation}\label{eq:goe}
\det\!\left(I-A\bar P_0\hat R\bar P_0A^*\right)=F_\mathrm{GOE}(4^{1/3}m).
\end{equation}
\end{prop}

The last equality is a version of the determinantal formula for $F_\mathrm{GOE}$ proved by
\citet{ferrariSpohn}, and which essentially goes back to \citet{sasamoto}:
\begin{equation}
  \label{eq:GOE}
  F_\mathrm{GOE}(m)=\det(I-P_0B_mP_0),
\quad {\rm
where}\quad
B_m(x,y)=\Ai(x+y+m).\end{equation}  This can be seen as follows.
Using the cyclic property of the determinant and  the \emph{reflection operator} $\sigma f(x)=f(-x)$ we may rewrite the
determinant in \eqref{eq:goe} as
\begin{equation}
  \begin{aligned}
    \det\!\left(I-\bar P_0\hat R\bar P_0\right)=\det\!\left(I-\sigma\bar P_0\hat R\bar P_0\sigma\right)
    =\det\!\left(I-P_0\sigma\hat R\sigma P_0\right),
  \end{aligned}\label{eq:preGOE}
\end{equation}
where we have used that $AA^*=\sigma^2=I$. On the other hand we have $\sigma\hat
R\sigma(\lambda,\tilde\lambda) =2^{-1/3}\Ai(2^{-1/3}(\lambda+\tilde\lambda+2m))$.
Performing the change of variables $\lambda\mapsto2^{1/3}\lambda$,
$\tilde\lambda\mapsto2^{1/3}\tilde\lambda$ in the Fredholm determinant shows that the
determinants in \eqref{eq:preGOE} equal $\det(I-P_0B_{4^{1/3}m}P_0)$.

\vs

The rest of the paper is organized as follows.  In Section \ref{sec:indirect} we give an
overview of the approach of \citet{johansson} explaining how Theorem \ref{thm:ai} can be
obtained indirectly using the connection of the Airy$_2$ process with last passage
percolation. Section \ref{sec:aiL} contains a brief introduction to relevant ideas of
Fredholm determinants and then provides a proof of Theorem \ref{thm:aiL}.  Section
\ref{sec:extras1} provides a short proof of Proposition \ref{thm:goe}. Finally, Section
\ref{sec:extras2} is devoted to the proof of Lemma \ref{eq:omega}, which essentially
amounts to asymptotic analysis involving the Airy function.

\paragraph{\bf Acknowledgements}
JQ and DR were supported by the Natural Science and Engineering Research Council of
Canada, and DR was supported by a Fields-Ontario Postdoctoral Fellowship. IC was supported
by NSF through the PIRE grant OISE-07-30136.  The authors thank Victor Dotsenko and
Konstantine Khanin for interesting and helpful discussions, and Kurt Johansson for several
references to the physics literature.  Part of this work was done during the Fields
Institute program ``Dynamics and Transport in Disordered Systems" and the authors would
like to thank the Fields Institute for its hospitality.

\section{Indirect derivation of Theorem \ref{thm:ai}
  through last passage percolation}
\label{sec:indirect}

As we mentioned in the introduction, \citet{johansson} presented an indirect proof of
Theorem \ref{thm:ai} by way of the PNG model. His idea was entirely correct, but in the
process of translating between the available results at the time, a factor of $4^{1/3}$
was lost. The purpose of this section is to explain Johansson's approach and account
  for the missing $4^{1/3}$.

  We consider the PNG model (which we define below) with two types of initial conditions
  (droplet and flat), and show that by coupling them to the same Poisson point process
  environment we can represent the one-point distribution for the flat case as the maximum
  of the interface in the droplet case. Asymptotics of this relationship leads to the
  identity in Theorem \ref{thm:ai}.

\begin{figure}
\begin{center}
\includegraphics[scale=.7]{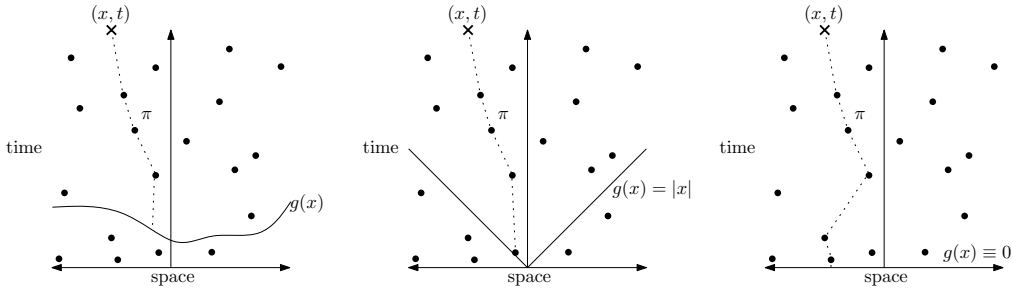}
\end{center}
\caption{The maximization problems coupled to the same Poissonian environment. Paths $\pi$ must be Lipschitz 1 functions of time and $T(\pi)$ represents a count of the number of Poisson points encountered by $\pi$. Left: A general function $g(x)$ represents the possible starting space-time starting location. Middle: The {\it droplet} geometry in which $g(x)=|x|$. Right: The {\it flat} geometry in which $g(x)\equiv 0$.}\label{LPP_Fig}
\end{figure}

Consider a space-time Poisson point process $P$ of intensity 2. Define a height function
above $x$ at time $t$ as
\begin{equation}
  h_g(x,t) = \max_{\pi:g\to (x,t)} T(\pi)
\end{equation}
where $g$ represents a space-time curve $(g(x),x)_{x\in \rr}$, $\pi$ is a Lipschitz 1
function of time (i.e., $|\pi(s)-\pi(s')|\leq |s-s'|$ for all $s,s'$), $\pi\!:g\to (x,t)$
means that $\pi$ starts at a point of the form $(g(x),x)$ and ends at the point $(x,t)$,
and $T(\pi)$ represents the sum of the number of Poisson points that $\pi$ touches. We
will specialize this definition to two cases. In the {\it droplet} geometry (for which we write
$h^{{\rm droplet}}$) we take $g=|x|$, hence we only consider paths
originating along a wedge. As a result the maximal path will always originate at the
origin $(0,0)$. In the {\it flat} geometry (for which we write $h^{{\rm flat}}$) we take
$g\equiv 0$, hence we consider Lipschitz paths starting in any spatial location
at time 0 and ending at $x$ at time $t$. This are illustrated in Figure \ref{LPP_Fig}.

Couple $P$ to another Poisson point process $\tilde P$ via $\tilde P(A) = P(\tau_{t} A)$, where for
any Borel set $A\in \rr^{2}$, $(y,s)\in \tau_t A$ if and only if $(-y,t-s)\in A$ (one
should think of this as a time-reversal of the Poisson point process where $s\mapsto t-s$
and $x\mapsto -x$). Let $\tilde{h}^{{\rm flat}}$ represent the flat geometry height
function built on the $\tilde P$ Poisson point process. Then the following relation holds
\begin{equation}
  \tilde{h}^{{\rm flat}}(t,0) = \max_{x\in \rr} h^{{\rm droplet}}(t,x).
\end{equation}

Asymptotic fluctuation statistics have been derived for both the droplet and flat
geometries and (up to justification of taking the limit inside the maximum, as done in
\cite{johansson} for a related model) the limiting statistics also respect the same
relationship above. Specifically \cite{prahoferSpohn} (see also \cite{borFerrPrah} for the
specific choices of scaling used below) shows that
\begin{equation}
  \lim_{t\to \infty} \frac{h^{{\rm droplet}}(t,t^{2/3}x) - 2t}{t^{1/3}} = \aip(x)-x^2.
\end{equation}
This implies that (up to the justifications mentioned above)
\begin{equation}
  \lim_{t\to\infty} \frac{\tilde{h}^{{\rm flat}}(t,0)-2t}{t^{1/3}} = \max_{x\in \rr}\left(\aip(x)-x^2\right).
\end{equation}

On the other hand, \cite{borodFerSas} shows that
\begin{equation}
  \lim_{t\to\infty} \frac{\tilde{h}^{{\rm flat}}(t,0)-2t}{t^{1/3}} = 2^{1/3}\mathcal{A}_1(0)
\end{equation}
where $\mathcal{A}_1$ is the Airy$_1$ process. Combining these two identities shows that
\begin{equation}
  \pp\big(\max_{x\in \rr}(\aip(x)-x^2)\leq m\big)=\pp\big(\mathcal{A}_{1}(0)\leq
  2^{-1/3}m\big) =F_\mathrm{GOE}(4^{1/3}m),
\end{equation}
where the last equality follows from work of Ferrari and Spohn \cite{ferrariSpohn} which shows that
$\pp(\mathcal{A}_1(0)<m) = F_\mathrm{GOE}(2m)$.

\section{Proof of Theorem \ref{thm:aiL}}\label{sec:aiL}
The operator in (\ref{eq:airyfd}) should be seen as a discrete version of the boundary
value problem operator $\Theta^g_{[\ell,r]}$. In particular, for $n>0$ let
$t_i=\ell+i(r-\ell)/n$, $i=0,\dotsc,n$, and define the discrete time boundary value
problem operator
\begin{equation}\label{discBVP}
\Theta^g_{n,[\ell,r]}=\bar P_{g(t_0)}e^{(t_0-t_1)H}\bar P_{g(t_1)}e^{(t_1-t_2)H}\dotsm
  e^{(t_{n-1}-t_n)H}\bar P_{g(t_n)}.
\end{equation}
The proof of Theorem \ref{thm:aiL} amounts to showing that, as $n$ goes to infinity, the
discrete operator converges to the limiting operator $\Theta^g_{[\ell,r]}$. This
convergence must be in a suitably strong sense to ensure the convergence of the Fredholm
determinants. Therefore, before turning to the proof of Theorem \ref{thm:aiL}, let us
briefly review some facts about Fredholm determinants, trace class operators and
Hilbert-Schmidt operators (see Section 2.3 in \cite{acq} for more details, a complete
treatment can be found in \cite{simon}).
  
Consider a separable Hilbert space $\ch$ and let $A$ be a bounded linear operator acting
on $\ch$. Let $|A|=\sqrt{A^*A}$ be the unique positive square root of the operator
$A^*A$. The \emph{trace norm} of $A$ is defined as $\|A\|_1= \sum_{n=1}^{\infty}\langle
e_n,|A|e_n\rangle$, where $\{e_n\}_{n\geq 1}$ is any orthonormal basis of $\ch$. We say that
$A\in\mathcal{B}_1(\ch)$, the family of \emph{trace class operators}, if
$\|A\|_1<\infty$. For $A\in\mathcal{B}_1(\ch)$, one can define the trace
$\tr(A)=\sum_{n=1}^{\infty}\langle e_n,A e_n\rangle$ and then the \emph{Hilbert-Schmidt
  norm} $\|A\|_2 =\sqrt{\tr(|A|^2)}$. We say that $A\in\cb_2(\ch)$, the family
\emph{Hilbert-Schmidt operators}, if $\|A\|_2<\infty$. The following lemma collects some
results which we will need in the sequel, they can be found in Chapters 1-3 of
\cite{simon}:

\begin{lem}\label{lem:fredholm}
  \mbox{}
  \begin{enumerate}[label=(\alph*),itemsep=4pt]
  \item $A\mapsto\det(I+A)$ is a continuous function on $\mathcal{B}_1(\ch)$. Explicitly,
    \[|\det(I+A)-\det(I+B)|\leq\|A-B\|_{1}\exp(\|A\|_1+\|B\|_1+1).\]
  \item If $A\in \mathcal{B}_1(\ch)$ and $A=BC$ with $B,C\in \mathcal{B}_2(\ch)$, then
    $\|A\|_1\leq\|B\|_2\|C\|_2$.
  \item If $\|A\|_\mathrm{op}$ denotes the operator norm of $A$ in $\ch$, then
    $\|A\|_\mathrm{op}\leq\|A\|_2\leq\|A\|_1$, $\|AB\|_1\leq\|A\|_\mathrm{op}\,\|B\|_1$ and $\|AB\|_2\leq\|A\|_\mathrm{op}\,\|B\|_2$.
  \item If $A\in\cb_2(\ch)$, then $\|A^*\|_2=\|A\|_2$. If $A$ has integral kernel $A(x,y)$,
    then
    \[\|A\|_2=\left(\int dx\,dy\,|A(x,y)|^2\right)^{1/2}.\]
  \end{enumerate}
\end{lem}

The proof of the continuum limit of \eqref{eq:airyfd} will follow easily from the
next proposition.

\begin{prop}\label{prop:theta}
  Assume $g\in H^1([\ell,r])$. Then the operators
  $K_{\Ai}-\Theta^g_{n,[\ell,r]}e^{(r-\ell)H}K_{\Ai}$ and
  $K_{\Ai}-\Theta^g_{[\ell,r]}e^{(r-\ell)H}K_{\Ai}$ are in $\cb_1(L^2(\rr))$, with
  $\|K_{\Ai}-\Theta^g_{n,[\ell,r]}e^{(r-\ell)H}K_{\Ai}\|_1$ bounded uniformly in
  $n$. Furthermore, for any fixed $\ell<r$ we have, writing $n_k=2^k$,
  \begin{equation}
    \lim_{k\to\infty}\|(K_{\Ai}-\Theta^g_{n_k,[\ell,r]}e^{(r-\ell)H}K_{\Ai})-(K_{\Ai}-\Theta^g_{[\ell,r]}e^{(r-\ell)H}K_{\Ai})\|_1=0.\label{eq:cvCont}
  \end{equation}
\end{prop}


The idea of the proof is the following. Just as done in the introduction for
$\Theta^g_{[\ell,r]}$, it is possible to use the Feynman-Kac and Cameron-Martin-Girsanov
formulas to write a formula for the kernel of $\Theta^g_{n_k,[\ell,r]}$ in terms of a path
integral with a killing potential enforced only at the dyadic mesh of times
$\{t_i\}_{i=1}^{n_k}$ (as opposed to being enforced at all times in $[\ell,r]$). If one
considers a parabolic barrier $g$ then the kernel for $\Theta^g_{n_k,[\ell,r]}$ is given
in terms of the probability of a Brownian bridge exceeding a fixed value at some time
$\{t_i\}_{i=1}^{n_k}$. This is compared to the analogous kernel for $\Theta^g_{[\ell,r]}$
given in terms of the probability of a Brownian bridge exceeding a fixed valued at any
time $t\in[\ell,r]$. As the mesh goes to zero, these two probabilities converge and hence
so do the kernels. This proves the proposition for parabolic $g$, and the extension to
$g\in H^1([\ell,r])$ then follows readily since $H^1$ is the Cameron-Martin space for
Brownian motion.

\begin{proof}[Proof of Proposition \ref{prop:theta}]
  We will first prove the result assuming that $g(s)=(s-\frac12(\ell+r)^2)^2$. Let
  $\varphi(x)=(1+x^2)^{1/2}$ and define the multiplication operator $Mf(x)=\varphi(x)f(x)$ (note that the choice
  of $\varphi$ is not particularly important and any strictly positive, polynomially growing function would do). To
  estimate the trace norm of $K_{\Ai}-\Theta^g_{[\ell,r]}e^{(r-\ell)H}K_{\Ai}$ we 
  use Lemma \ref{lem:fredholm} to write
  \begin{equation}
    \big\|K_{\Ai}-\Theta^g_{[\ell,r]}e^{(r-\ell)H}K_{\Ai}\big\|_1\leq
    \big\|(e^{-(r-\ell)H}-\Theta^g_{[\ell,r]})M\big\|_2\|M^{-1}e^{(r-\ell)H}K_{\Ai}\big\|_2.\label{eq:normDec}
  \end{equation}
  For the second Hilbert-Schmidt norm above we have by \eqref{eq:eLHK} that
  \begin{equation}
    \begin{split}
      \|M^{-1}e^{(r-\ell)H}K_{\Ai}\big\|^2_2&=\int_{\rr^2}dx\,dy\int_{(-\infty,0]^2}d\lambda\,d\tilde\lambda\,
      \varphi(x)^{-2}e^{(\lambda+\tilde\lambda)(r-\ell)}\Ai(x-\lambda)\Ai(y-\lambda)\\
      &\hspace{3in}\cdot\Ai(x-\tilde\lambda)\Ai(y-\tilde\lambda)\\
      &=\int_{-\infty}^\infty dx\int_{-\infty}^0d\lambda\,\varphi(x)^{-2}e^{2\lambda
        (r-\ell)}\Ai(x-\lambda)^2\leq c\,(2(r-\ell))^{-1}\|\varphi^{-1}\|_2^2,
    \end{split}\label{eq:sndHS}
  \end{equation}
  where $c=\max_{x\in\rr}\Ai(x)^2<\infty$.  
     
  Now we consider the first norm on the right side of \eqref{eq:normDec}. Shifting
  time by $-(\ell+r)/2$ in the definition of $\Theta^g_{[\ell,r]}$ it is clear that this
  operator equals $\Theta^{\tilde{g}}_{[-L,L]}$, where $L=(r-\ell)/2$ and
  $\tilde{g}(s)=s^2$. Using the formula for $\Theta^{\tilde g}_{[-L,L]}(x,y)$ given in
  Theorem \ref{thm:thetaLgen} we get
  \begin{equation}
    \Theta^g_{[\ell,r]}(x,y)=\frac{e^{-(x-y)^2/8L-(x+y)L+2L^3/3}}{\sqrt{8\pi L}}
    \pp_{\hat b(-L)=x-L^2,\hat b(L)=y-L^2}\!\left(\hat b(s)\leq 0\text{ on }[-L,L]\right).\label{eq:ThetaL2}
  \end{equation}
  Similarly, the kernel of $e^{-(r-\ell)H}=e^{-2LH}$ equals the above one with the
  probability replaced by 1, and hence
  \begin{multline}\label{eq:diff1}
    \big(e^{-(r-\ell)H}-\Theta^{g}_{[\ell,r]}\big)M(x,y)=\frac{e^{-(x-y)^2/8L-(x+y)L+2L^3/3}}{\sqrt{8\pi L}}\varphi(y)\\
    \cdot\pp_{\hat b(-L)=x-L^2,\hat b(L)=y-L^2}\!\left(\hat b(s)\geq 0\text{ for some
      }s\in[-L,L]\right).
  \end{multline}
  Using a known Brownian bridge formula (see for example page 67 in \cite{handbookBM}),
  the latter crossing probability equals $e^{-(x-L^2)(y-L^2)/2L}$ if $x\leq L^2,y\leq L^2$
  and 1 otherwise, and therefore
  \begin{multline}
    \|(e^{-(r-\ell)H}-\Theta^{g}_{[\ell,r]})M\|_2^2= \frac{1}{8\pi
      L}\int_{\rr^2\setminus(-\infty,0]^2}dx\,dy
    \,\big[e^{-(x-y)^2/8L-(x+y)L-4L^3/3}\big]^2\varphi(y+L^2)^2\\
    +\frac{1}{8\pi L}
    \int_{(-\infty,0]^2}dx\,dy\,\big[e^{-(x+y)^2/8L-(x+y)L-4L^3/3}\big]^2\varphi(y+L^2)^2,\label{eq:bdHS}
  \end{multline}
  where we have performed the change of variables $x\mapsto x+L^2$, $y\mapsto y+L^2$. Both
  Gaussian integrals can be easily seen to be finite, so we have shown that
  $(e^{-(r-\ell)H}-\Theta^g_{[\ell,r]})M\in\cb_2(L^2(\rr))$. Using this with
  \eqref{eq:normDec} and \eqref{eq:sndHS} it follows that
  $K_{\Ai}-\Theta^g_{[\ell,r]}e^{(r-\ell)H}K_{\Ai}$ is in $\cb_1(L^2(\rr))$.

  Next we observe that we can shift time and apply the Feynman-Kac and
  Cameron-Martin-Girsanov formulas directly on $\Theta^g_{n,[\ell,r]}$ ($n$ times) exactly
  as we did for $\Theta^g_{[\ell,r]}$, and it is not hard to check that we get a formula
  analogous to \eqref{eq:diff1}:
  \begin{multline}\label{eq:diff2}
    \big(e^{-(r-\ell)H}-\Theta^g_{n,[\ell,r]}\big)M(x,y)=\frac{1}{\sqrt{8\pi L}}
    e^{-(x-y)^2/8L-(x+y)L+2L^3/3}\varphi(y)\\
    \cdot\pp_{\hat b^n(-L)=x-L^2,\hat b^n(L)=y-L^2}\!\left(\hat b^n(t^n_i)\geq 0\text{ for
        some }i\in\{0,\dotsc,n\}\right).
  \end{multline}
  where $\hat b^n$ is now a discrete time random walk with Gaussian jumps with mean 0 and
  variance $2L/n$, started at time $-L$ at $x-L^2$, conditioned to hit $y-L^2$ at time
  $L$, and jumping at times $t^n_i=-L+2iL/n$, $i\geq0$. 
  A simple coupling argument (see the next paragraph) shows that the last probability is
  less than the corresponding one for the Brownian bridge, and thus we obtain for
  $\|K_{\Ai}-\Theta^g_{n,[\ell,r]}e^{(r-\ell)H}K_{\Ai}\|_1$ the same bound as the one we get for
  $\|K_{\Ai}-\Theta^g_{[\ell,r]}e^{(r-\ell)H}K_{\Ai}\|_1$ from \eqref{eq:bdHS}. This bound is, in
  particular, independent of $n$.

  Finally, in order to prove \eqref{eq:cvCont} we couple the Brownian bridge $\hat b$ and
  the conditioned random walk $\hat b^{n_k}$ by simply letting $\hat
  b^{n_k}(t^{n_k}_i)=\hat b(t^{n_k}_i)$ for each $i=0,\dotsc,{n_k}$. Since the Brownian
  bridge hits the positive half-line whenever the conditioned random walk does, it is
  clear that
  \begin{multline}
    \left|\big(e^{-(r-\ell)H}-\Theta^g_{[\ell,r]}\big)M-\big(e^{-(r-\ell)H}-\Theta^{g}_{n_k,[\ell,r]}\big)M\right|\!(x,y)\\
    =\frac{e^{-(x-y)^2/8L-(x+y)L+2L^3/3}}{\sqrt{8\pi L}}\varphi(y)q_{n_k}(x,y),
  \end{multline}
  where $q_{n_k}(x,y)$ is the probability that the Brownian bridge $\hat b(s)$ hits the
  positive half-line for $s\in[-L,L]$ but not for any
  $s\in\{t^{n_k}_0,\dotsc,t^{n_k}_{2{n_k}}\}$.  Since every point is regular for
  one-dimensional Brownian motion, $q_{n_k}(x,y)\searrow0$ as $k\to\infty$ for every fixed
  $x,y$, and thus by the monotone convergence theorem we deduce that
  $\|(e^{-(r-\ell)H}-\Theta^g_{[\ell,r]})M-(e^{-(r-\ell)H}-\Theta^{g}_{n_k,[\ell,r]})M\|_2\to0$
  as $k\to\infty$. Using \eqref{eq:sndHS} and a decomposition analogous to
  \eqref{eq:normDec} yields \eqref{eq:cvCont}.

  To extend the result to $g\in H^1([\ell,r])$ we note that everything in the above
  argument deals with properties of a Brownian motion $b(s)$ killed at the positive
  half-line. In the general case we will have by Theorem \ref{thm:thetaLgen} a Brownian
  motion $b(s)$ killed at the boundary $g(s)-s^2$ or, equivalently, a process $\tilde
  b(s)=b(s)-g(s)+s^2$ killed at the positive half-line. Using the Cameron-Martin-Girsanov
  Theorem we can rewrite the probabilities for $\tilde b(s)$ in terms of probabilities for
  $b(s)$. Since $g(s)$ is a deterministic function in $H^1([\ell,r])$, the Radon-Nikodym
  derivative of $\tilde b(s)$ with respect to $b(s)$ has finite second moment, and thus by
  using the Cauchy-Schwarz inequality we get the first two statements in the result from
  the above arguments. The convergence in \eqref{eq:cvCont} follows as well from the above
  arguments because it only depends on almost sure properties of the corresponding
  Brownian motion.
\end{proof}

\begin{proof}[Proof of Theorem \ref{thm:aiL}]
  With the notation introduced before Proposition \ref{prop:theta} we have
  \begin{equation}
    \pp\!\left(\aip(t_0)\leq g(t_0),\dotsc,\aip(t_n)\leq g(t_{n_k})\right)
    =\det\!\left(I-K_{\Ai}+\Theta^g_{{n_k},[\ell,r]}e^{(r-\ell)H}K_{\Ai}\right),\label{eq:detn}
  \end{equation}
  where $n_k=2^k$. Since the Airy${}_2$ process has a version with continuous paths (see
  \cite{johansson,corwinHammond,quastelRemAiry1}), the probability above converges to
  $\pp\!\left(\aip(t)\leq g(t)\text{ for }t\in[\ell,r]\right)$ as $k\to\infty$. The result
  now follows from Proposition \ref{prop:theta} and Lemma
  \ref{lem:fredholm}, which imply that
  \[\lim_{k\to\infty}\det\!\big(I-K_{\Ai}+\Theta^g_{n_k,[\ell,r]}e^{(r-\ell)H}K_{\Ai}\big)
  =\det\!\big(I-K_{\Ai}+\Theta^g_{[\ell,r]}e^{(r-\ell)H}K_{\Ai}\big).\qedhere\]
\end{proof}

\section{Proof of Proposition \ref{thm:goe}}\label{sec:extras1}

Since $K_{\Ai}$ is the projection onto the negative (generalized) eigenspace of $H$ (see Remark
\ref{airyrem}), we have
\begin{equation}
e^{LH}K_{\Ai}(x,z)=\int_{-\infty}^0 d\lambda\,e^{\lambda
  L}\Ai(x-\lambda)\Ai(z-\lambda)\,d\lambda.\label{eq:eLHK}
\end{equation}
Then, recalling that the Airy transform is given by $Af(x)=\int_{-\infty}^\infty
d\lambda\Ai(x-\lambda)f(\lambda)$, we can write
\begin{equation}
  e^{LH}K_{\Ai}R_Le^{LH}K_{\Ai}=A\bar P_0\hat R_L\bar P_0A^*,\label{eq:dec}
\end{equation}
where
\begin{equation}
  \hat R_L(\lambda,\tilde\lambda)=\frac{1}{\sqrt{8\pi L}}\int_{\rr^2}d\tilde
  z\,dz\,e^{-(z+\tilde
    z-2m-2L^2)^2/8L-(z+\tilde z)L+(\lambda+\tilde\lambda)L+2L^3/3}
  \Ai(z-\lambda)\Ai(\tilde z-\tilde\lambda).\label{eq:hatRL}
\end{equation}
Applying the change of variables
$2u=z+\tilde z$, $2v=z-\tilde z$, we get
\begin{multline}
  \hat R_L(\lambda,\tilde\lambda)=\frac{1}{\sqrt{2\pi L}}\int_{\rr^2}
  du dv \,e^{-\frac{(u-m-L^2)^2}{2L}-2uL+(\lambda+\tilde\lambda)L+\tfrac23 L^3}\Ai(u+v-\lambda)\Ai(u-v-\tilde\lambda).
\end{multline}
Using the formula
\begin{equation}
\int_{-\infty}^\infty dx\Ai(a+x)\Ai(b-x)=2^{-1/3}\Ai(2^{-1/3}(a+b))\label{eq:airySq}
\end{equation}
(see, for example, (3.108) in \cite{valleeSoares}), the $v$ integral equals
$2^{-1/3}\Ai(2^{-1/3}(2u-\lambda-\tilde\lambda))$. Therefore
\begin{align}
  \hat R_L(\lambda,\tilde\lambda)&=\frac{2^{-1/3}}{\sqrt{2\pi L}}\int_{-\infty}^\infty
  du\,e^{-(u-m-L^2)^2/2L-2uL+(\lambda+\tilde\lambda)L+2L^3/3}\Ai(2^{-1/3}(2u-\lambda-\tilde\lambda))\\
  &=\frac{2^{-1/3}}{\sqrt{2\pi L}}\frac{1}{2\pi\I}\int_\Gamma dt\int_{-\infty}^\infty du\,
  e^{-(u-m-L^2)^2/2L-2uL+(\lambda+\tilde\lambda)L+2L^3/3+t^3/3-2^{-1/3}t(\lambda+\tilde\lambda-2u)},
\end{align}
where in the second equality we have used the contour integral representation of the Airy
function, $\Ai(x)=\frac{1}{2\pi\I}\int_\Gamma dt\,e^{t^3/3-tx}$ with $\Gamma=\{c+\I
s\!:s\in\rr\}$ and $c$ any positive real number. The $u$ integral is just a Gaussian
integral, and computing it we get
\[\hat R_L(\lambda,\tilde\lambda)=\frac{2^{-1/3}}{2\pi\I}\int_{\Gamma}dt\,
e^{t^3/3+2^{1/3}Lt^2+[4^{1/3}L^2+2^{-1/3}(\lambda+\tilde\lambda-2m)]t+2L^3/3+(\lambda+\tilde\lambda)L}.\]
Now we perform the change of variables $t=s-2^{1/3}L$ to obtain
\begin{equation}
\hat R_L(\lambda,\tilde\lambda)=\frac{2^{-1/3}}{2\pi\I}\int_{\Gamma'}ds\,
e^{s^3/3-2^{-1/3}(2m-\lambda-\tilde\lambda)s}
=2^{-1/3}\Ai(2^{-1/3}(2m-\lambda-\tilde\lambda))
\end{equation}
(here the contour $\Gamma'$ is simply
$\Gamma$ shifted by $2^{1/3}L$, so the integral still gives an Airy function). Note how all
the terms involving $L$ have canceled.

\section{Proof of Lemma \ref{eq:omega}}\label{sec:extras2}


The proof of this result amounts to asymptotic analysis of integrals involving the Airy function. The following well-known estimates for the Airy function (see
(10.4.59-60) in \cite{abrSteg}) go a long way in the proof:
\begin{equation}
  |\!\Ai(x)|\leq Ce^{-\frac23x^{3/2}}\quad\text{for $x>0$},\qquad
  |\!\Ai(x)|\leq C\quad\text{ for $x\leq0$}.\label{eq:airybd}
\end{equation}
However, at one important point the above bounds for $x\leq 0$ will prove insufficient,
and it will become necessary to utilize a representation \eqref{eq:airyG} for $\Ai(x)$
which splits it into complex oscillations of opposite phase. Then, following standard
methods of asymptotics for oscillatory integrals (i.e., shifting real contours up and down
to turn oscillations into exponential decay), we will achieve our bounds needed to
complete the proof of this lemma.

We will use the following version of Laplace's method, which we state without proof (see,
for instance, \cite{erdelyi}):

\begin{lem}\label{lem:laplace}
  Let
  \[I(M)=\int_\Omega dx\,f(x)e^{\varphi(x)M},\] where $\Omega\subseteq\rr^n$ is a
  (possibly unbounded) open polygonal domain and $f$ and $\varphi$ are smooth functions
  defined on $\overline\Omega$. Assume that the local maxima of $\varphi$ are attained at
  a finite subset $\{x_1,\dotsc,x_n\}$ of $\overline\Omega$. Then there is a constant
  $C>0$ such that
  \[|I(M)|\leq C\sum_{k=1}^nM^{-\kappa_i}|f(x_i)|e^{\varphi(x_i)M}\] for large enough $M$,
  where $\kappa_i=(n+1)/2$ if $x_i\in\p\Omega$ and $\kappa_i=n/2$ if $x_i\in\Omega$.
\end{lem}





We write  $ \widetilde \Omega_L =  \widetilde \Omega^1_L - \widetilde \Omega^2_L$ where
 \begin{align}
   \widetilde \Omega^1_L&=e^{LH}K_{\Ai}\left(R_L-\bar P_{m+L^2}R_L\bar
     P_{m+L^2}\right)e^{LH}K_{\Ai},\\
   \widetilde \Omega^2_L&=e^{LH}K_{\Ai}\left(e^{-2LH}-\bar P_{m+L^2}e^{-2LH}\bar
     P_{m+L^2}\right)e^{LH}K_{\Ai},
\end{align}

The proof of Lemma \ref{eq:omega} is contained in the next two lemmas.

\begin{lem}\label{lem:tildeRL2to0}
  \[\|\widetilde \Omega^1_L\|_1\xrightarrow[L\to\infty]{}0.\]
\end{lem}

\begin{proof}
  We proceed as in \eqref{eq:dec} and factorize $\widetilde \Omega^1_L$ as
  \begin{equation}
    \widetilde \Omega^1_L=A\hat\Omega^1_LA^*,\label{eq:formR2}
  \end{equation}
  where
  \begin{multline}\label{eq:hatOmegaPre}
    \hat\Omega^1_L(\lambda,\tilde\lambda)=
    \frac{1}{\sqrt{8\pi L}}\int_{\widetilde D}dz\,d\tilde z\, e^{-(z+\tilde
      z-2m-2L^2)^2/8L-(z+\tilde z)L+2L^3/3+(\lambda+\tilde\lambda)L}\\
    \cdot\Ai(z-\lambda)\Ai(\tilde z-\tilde\lambda)\uno{\lambda,\tilde\lambda\leq0},
  \end{multline}
  with $\widetilde D=\rr^2\setminus(-\infty,m+L^2]^2$. 
  Using the Plancherel formula for the Airy transform $\int\!f^2=\int(Af)^2$, we have
  $\|A\|_{\rm op}=\|A^*\|_{\rm op}=1$, 
  so by Lemma \ref{lem:fredholm} it will be enough to show that
  \begin{equation}
    \label{eq:normbd}
    \|\hat\Omega^1_L\|_1\xrightarrow[L\to\infty]{}0.
  \end{equation}

  Performing the change of variables $z=L^2(1-w)+m$, $\tilde z=L^2(1-\tilde w)+m$ in
  \eqref{eq:hatOmegaPre} the kernel becomes
  \begin{multline}\label{eq:hatRL2-2}
    \hat \Omega^1_L(\lambda,\tilde\lambda) = \frac{L^{7/2}
      e^{(\lambda+\tilde\lambda)L-2mL}}{\sqrt{8\pi}} \int_{D}dw\,d\tilde w\,
    e^{L^3f(w,\tilde w)}\Ai(L^2(1-w)+m-\lambda)\\
    \cdot\Ai(L^2(1- \tilde w)+m-\tilde\lambda)\uno{\lambda,\tilde\lambda\leq0},
  \end{multline}
  where $D=\rr^2\setminus[0,\infty)^2$ and $f(w,\tilde w) = \frac{-(w+\tilde
    w)^2}{8} + (w+\tilde w) -\tfrac{4}{3}$.

  We split the region $D$ into the union of three disjoint regions of pairs $(w,\tilde
  w)$: $D_1=\{w\leq 1, \tilde w \leq 1\} - \{0\leq w\leq 1, 0\leq \tilde w \leq 1\}$,
  $D_2=\{w\leq 0, \tilde w\geq 1\}$ and $D_2'=\{w\geq 1,\tilde w\leq 0\}$. By the triangle
  inequality we can bound $\|\hat \Omega^1_L(\lambda,\tilde \lambda)\|_1$ by the sum of
  the trace norms of the operators obtained by restricting the integral in
  \eqref{eq:hatRL2-2} to each of the regions $D_1$, $D_2$ and $D_2'$. We will write
  $\hat\Omega^1_L\uno{D_1}$ for the operator restricted to $D_1$, with the analogous
  notation for the other regions. Notice that, due to
  the symmetry of our formula, we do not need to bound $\|\hat\Omega^1_L\uno{D_2'}\|_1$,
  as it satisfies the same bound as $\|\hat\Omega^1_L\uno{D_2}\|_1$.

  Let us focus first on the operator restricted to $D_1$, which is the simplest case
  because $1-w\geq 0$ and $1-\tilde w\geq 0$. We write $\hat\Omega^1_L=G_1G_2G_3$ with
  \begin{equation}
    \begin{gathered}
      G_1(\lambda,w)=e^{(\lambda-m)L}\Ai(L^2(1-w)+m-\lambda)\uno{\lambda\leq0,w\leq
        1},\quad
      G_2(w,\tilde w)=\frac{L^{7/2}}{\sqrt{8\pi}}e^{L^3f(w,\tilde w)}\uno{(w,\tilde w)\in D_1},\\
      G_3(\tilde w,\tilde\lambda)=e^{(\tilde\lambda-m)L}\Ai(L^2(1-\tilde
      w)+m-\tilde\lambda)\uno{\tilde\lambda\leq0,\tilde w\leq1}
    \end{gathered}
  \end{equation}
  and use Lemma \ref{lem:fredholm} to estimate
  $\|\hat\Omega^1_L\|_1\leq\|G_1\|_2\|G_2\|_2\|G_3\|_2$. We begin with $G_1$ and assume
  first that $w\leq1+mL^{-2}$, so that using \eqref{eq:airybd} have
  \begin{multline}
    \|G_1\bar P_{1+mL^{-2}}\|^2_2=\int_{-\infty}^0d\lambda\int_{-\infty}^{\min\{1,1+mL^{-2}\}}dw\,e^{2(\lambda-m)L}\Ai(L^2(1-w)+m-\lambda)^2\\
    \leq Ce^{-2mL}\int_0^\infty d\lambda\int_0^\infty dw\,e^{-2\lambda-4/3(L^2w+\lambda)^{3/2}}
    \leq C'e^{-2mL}.
  \end{multline}
  Likewise, using the other bound in \eqref{eq:airybd} it is easy to get
  $\|G_1P_{1+mL^{-2}}\|^2_2\leq C|m|e^{-2mL}$. Therefore $\|G_1\|_2\leq Ce^{-mL}$,
  and of course the same bound works for $G_3$. On the other hand, $f$ attains its maximum
  on $D_1$ at the points $(1,0)$ and $(0,1)$, where its value is $-\frac{11}{24}$. Lemma
  \ref{lem:laplace} then allows to conclude that
  \begin{equation}
    \|G_2\|^2_2=\int_{D_1}dw\,d\tilde w\,\frac{L^7}{8\pi}e^{2L^3f(w,\tilde w)}
    \leq CL^7(L^3)^{-3/2}e^{-11L^3/12}.
  \end{equation}
  Putting the three bounds together we deduce that
  \begin{equation}
    \label{eq:SW}
    \|\hat\Omega^1_L\uno{D_1}\|_1\leq e^{-CL^3}
  \end{equation}
  for some $C>0$.

  Let us now turn to the trace norm of the operator restricted to $D_2$ (and hence also to
  $D_2'$). This bound is slightly harder owing to the fact that one of the Airy functions
  is oscillatory (rather than rapidly decaying) in this region. As readily derived from
  the contour integral representation of the Airy function by deforming the contour and
  performing a change of variables, $\Ai(\cdot)$ may alternatively be expressed as
  \begin{equation}
    \Ai(x)=\re\!\left[\frac{\sqrt{-x}}{2\pi\I}\int_{\Gamma}ds\,\exp(\I(-x)^{3/2} (-s+s^3/3))\right],
  \end{equation}
  where $\Gamma$ is the contour $\{s = a+ b(a)\I:a>0\}$ with $b(a) =
  (a-1)\sqrt{\frac{a+2}{3a}}$. This contour is the steepest descent contour for $f(s) =
  \I(-s+s^3/3)$ and has the property that $\im f(s) =\im f(s_0)=-2/3$, where $s_0=1$ is a
  critical point of $f$. Along $\Gamma$ we can write $f(s) = -2/3\I + g(s)$ where $g(s)$
  is real valued, $g(s_0)=0$ and $g(s)$ decays to $-\infty$ monotonically and
  quadratically with respect to $|s-s_0|$.  Thus we may also write
  \begin{equation}\label{eq:airyG}
    \Ai(x) = \tfrac{1}{2}(G(-x) + \overline{G(-x)})
  \end{equation}
  where
  \[
  G(x) = \exp(-\tfrac{2}{3}x^{3/2}\I)\frac{\sqrt{x}}{2\pi\I}\int_{\Gamma}ds\,\exp(x^{3/2}
  g(s)).
  \]
  This expansion of the Airy function is the key to our oscillatory asymptotics.

  By applying the change of variables $w= 0 + L^{-3/2} v$ and $\tilde w = 4+ L^{-3/2}
  \tilde v$ the integral we wish to bound is given by
  \begin{equation}\label{eq:Rhatscaled}
    \hat \Omega^1(\lambda,\tilde\lambda) = \frac{L^{1/2}
      e^{(\lambda+\tilde\lambda)L}}{\sqrt{8\pi}} \int_{-\infty}^{0}dv\,h_L(v)
    \int_{-3 L^{3/2}}^{\infty} d\tilde v\,e^{-\frac{(v+\tilde v)^2}{8}}
    \Ai(-3L^2-L^{1/2}\tilde v+m-\tilde\lambda)\uno{\lambda,\tilde\lambda\leq0},
  \end{equation}
  where
  \[
  h_L(v)=e^{\tfrac{2}{3} L^3} \Ai(L^2 - L^{1/2} v +m-\lambda).
  \]
  We rewrite this as
  \begin{equation}
    \hat \Omega_L^1=H_1H_2\label{eq:H1H2}
  \end{equation}
  with
  \begin{gather}
    H_1(\lambda,v)= \frac{L^{1/2}}{\sqrt{8\pi}}e^{\lambda L-v^2/16}h_L(v)\uno{\lambda,v\leq0},\\
    H_2(v,\tilde\lambda)= e^{\tilde\lambda L-v^2/16}\int_{-3 L^{3/2}}^{\infty} d\tilde
    v\,e^{-\frac{2v\tilde v+\tilde v^2}{8}}
    \Ai(-3L^2-L^{1/2}\tilde v+m-\tilde\lambda)\uno{v,\tilde\lambda\leq0}.
  \end{gather}

  We will focus on the last integral in $\tilde v$ and prove that it is bounded by
  $e^{-CL^{3}}$ for some $C>0$. As the Airy function is bounded on the real axis, we
  readily find that due to the Gaussian term, we may cut our integral outside of a region
  $R_{\delta}=(-\delta L^{3/2}, \delta L^{3/2})$ by introducing an error
  of order $e^{-C'L^3}$. Thus we may restrict our attention to $R_{\delta}$.

  Using the expansion given by equation \eqref{eq:airyG}, the $\tilde v$ integral may be
  written as $\tfrac{1}{2}(I_L+\overline{I_L})$ with
  \[I_L=\int_{-\delta L^{3/2}}^{\delta L^{3/2}} d\tilde v\, e^{-\frac{2v\tilde v+\tilde
      v^2}{8}} G(3L^2+L^{1/2}\tilde v-m+\tilde \lambda).\] We wish to show that
  $|I_L|=|\overline{I_L}|\leq e^{-CL^{3}}$. For simplicity we set $v=\tilde\lambda=m=0$,
  though the argument below does not rely on this assumption and applies equally well for
  all $\tilde\lambda\leq0$, $v\leq0$ and $m\neq0$ as necessary. Under this simplification,
  and performing a change of variables from $\tilde v$ to $r$ by setting
  \begin{equation}
    L^3 r = (3L^{2} + L^{1/2} \tilde v)^{3/2},
  \end{equation}
  we obtain
  \begin{equation}
    I^{1}_L = \frac{2}{3} L^{3/2} \int_{3^{3/2}-\delta'}^{3^{3/2}+\delta'} dr\, r^{-1/3}
    e^{-\frac{L^{3}}{8}{(r^{2/3}-3)^2}}G((L^3 r)^{2/3}).
  \end{equation}
  Since we can consider an arbitrary $\delta$ before the change of variables, we can
  likewise consider an arbitrary $\delta'>0$ for which to bound $I_L$. Plugging in the
  expression for $G$ we get
  \begin{equation}
    I_L = \frac{L^{5/2}}{3\pi}\int_{3^{3/2}-\delta'}^{3^{3/2}+\delta'} dr\,
    e^{-L^{3}\left[\frac{(r^{2/3}-3)^2}{8} -\tfrac{2}{3} r \I\right]}\int_{\Gamma}\exp(L^3 r g(s))ds.
  \end{equation}
  Observe that this integrand is analytic in $r$. Thus by Cauchy's theorem, rather than
  integrating from $3^{3/2}-\delta'$ to $3^{3/2}+\delta'$ along the real axis, we may do
  so along any other curve between these points. Due to the properties of $g(s)$ along
  $\Gamma$, as long as $\re(r)>0$ we have that
  \begin{equation}
    \left|\int_{\Gamma}ds\exp(L^3 r g(s))\right| \leq \int_{\Gamma}ds\exp(L^3 \re(r) g(s)),
  \end{equation}
  which is certainly a bounded function of $r$ for $\re(r)>0$. The decay of the integrand
  is thus controlled by
  \begin{equation}\label{eq:realpart}
    \re\!\left(-\left[\tfrac18 {(r^{2/3}-3)^2} -\tfrac{2}{3} r \I\right] \right).
  \end{equation}
  Informed by this we may deform the $r$ integration contour to the contour $B=B_1\cup
  B_2\cup B_3$ where $B_1 = \{3^{3/2}-\delta' + \I y: y\in [0,\eta]\}$, $B_2 = \{x +
  \I\eta: x\in [3^{3/2}-\delta',3^{3/2}+\delta']\}$ and $B_3 = \{3^{3/2}+\delta' + \I y:
  y\in [0,\eta]\}$. It is an exercise in basic complex analysis to see that one can choose
  $\eta$ in such a way that, along the contour $B$, \eqref{eq:realpart} stays bounded
  below a constant $-C$ for $C>0$. This implies that the exponential is bounded by
  $e^{-CL^3}$ along that curve and hence for the entire integral we get $|I_L|\leq
  e^{-CL^3}$ for some $C>0$. Going back to the definition of $H_2$ this implies that
  \begin{equation}
    \|H_2\|^2_2=\frac14\int_{-(\infty,0]^2}dv\,d\tilde\lambda\,e^{2\tilde\lambda L-v^2/8}(I_L+\overline{I_L})^2
    \leq e^{-CL^3}.
  \end{equation}

  Returning to \eqref{eq:H1H2}, it now suffices by Lemma \ref{lem:fredholm} to prove that
  $\|H_1\|_2$ does not grow like $e^{CL^3}$. This follows readily by integration using
  \eqref{eq:airybd}, which implies that
  \begin{equation}\label{eq:hL}
    \log (h_L(v)) \approx \tfrac{2}{3} L^3 - \tfrac{2}{3}(L^2 - L^{1/2}v + \lambda)^{3/2} \approx C L^{3/2} v
  \end{equation}
  for some fixed $C>0$. Note how the $L^3$ terms perfectly cancel. This finishes showing that
  $\|\hat\Omega^1_L\uno{D_2}\|_1\leq e^{-CL^3}$. As noted before, we may likewise develop a bound
  for $\|\hat\Omega^1_L\uno{D_2'}\|_1$. Putting this together with \eqref{eq:SW} gives
  \eqref{eq:normbd}, which finishes the proof.
\end{proof}


\begin{lem}\label{lem:enie0}
  \[\|\widetilde\Omega^2_L\|_1\xrightarrow[L\to\infty]{}0.\]
\end{lem}

\begin{proof}
  The proof of this result is the same as that of the previous lemma. Using the definition
  of $\widetilde \Omega^2_L$ and factorizing as in the above proof
  we get
  \[\widetilde \Omega^2_L=A\hat\Omega^2_LA^*\]
  with
  \begin{multline}
    \hat \Omega^2_L(\lambda,\tilde\lambda)
    =\frac{e^{-2mL+(\lambda+\tilde\lambda)L}}{\sqrt{8\pi L}}\int_{D}d\tilde w\,dw\,e^{-(w-\tilde w)^2/8L+(w+\tilde
       w)L-4L^3/3}\\
     \cdot\Ai(-w+L^2-\lambda+m)\Ai(-\tilde w+L^2-\tilde\lambda+m)\uno{\lambda,\tilde\lambda\leq0}.
  \end{multline}
  Applying the change of variables $w\mapsto L^2 w$ and
  $\tilde w\mapsto L^2 \tilde w$ the kernel becomes
  \begin{multline}
    \hat\Omega^2_L(\lambda,\tilde\lambda) = \frac{L^{7/2} e^{(\lambda+\tilde\lambda)L/2-2mL}}{\sqrt{8\pi}} \int_{D}dw\,d\tilde w\,
    e^{L^3\tilde f(w,\tilde w)}\Ai(L^2(1-w)+m-\lambda)\\
    \cdot\Ai(L^2(1-\tilde w)+m-\tilde\lambda)\uno{\lambda,\tilde\lambda\leq0},
  \end{multline}
  where $\tilde f(w,\tilde w) = \frac{-(w-\tilde w)^2}{8} + (w+\tilde w)
  -\tfrac{4}{3}$. Note the similarity with \eqref{eq:hatRL2-2}, the only difference being
  that in $\tilde f$ we have a term $-(w-\tilde w)^2/8$ instead of $-(w+\tilde w)^2/8$.

  As in the above proof we need to bound $\|\hat \Omega^2_L\|_1$, and to that end we split
  $D$ into the same three regions $D_1$, $D_2$ and $D_2'$. The operator restricted to
  $D_1$ is easy to bound, exactly as before. On $D_2$ (and thus also on $D_2'$) we can
  repeat the same argument as before. The only diference is that, when we apply the change
  of variables $w= 0 + L^{-3/2} v$ and $\tilde w = 4+ L^{-3/2} \tilde v$, the function
  $h_L(v)$ in the resulting integral in \eqref{eq:Rhatscaled} is now multiplied by $e^{2v
    L^{3/2}}$, coming from the difference between $\tilde f$ and the function $f$ defined
  after \eqref{eq:hatRL2-2}. This change does not affect the bound on the $\tilde v$
  integral ($I_L$ and $\overline{I_L}$ in the above proof). It is straightforward to check
  that the rest of the proof is not affected either (note in fact that the only place
  where the definition of $h_L(v)$ is used is \eqref{eq:hL}, and the approximation there
  is still valid).
\end{proof}

\printbibliography[heading=apa]

\end{document}